\documentclass[12pt]{article}
\usepackage{amssymb}
\usepackage{amsmath}
\usepackage{amsthm}
\usepackage{authblk}
\usepackage{cancel}
\usepackage{caption}
\usepackage{subcaption}
\usepackage{cite}
\usepackage{CJK}
\usepackage{color}
\usepackage{enumerate}
\usepackage{esint}
\usepackage{fancyhdr}
\usepackage{fancyref}
\usepackage{graphics}
\usepackage{graphicx}
\usepackage[colorlinks = true,
            linkcolor = blue,
            urlcolor  = blue,
            citecolor = green,
            anchorcolor = violet]{hyperref}
\graphicspath{ {images/} }
\usepackage{mathtools}
\usepackage{mathrsfs}
\usepackage{pdfsync}
\usepackage{theoremref}
\usepackage{tikz}

\newtheorem{thm}{Theorem}[section]
\newtheorem{defn}[thm]{Definition}
\newtheorem{rk}[thm]{Remark}
\newtheorem{prop}[thm]{Proposition}
\newtheorem{lem}[thm]{Lemma}
\newtheorem{subl}[thm]{Sublemma}
\newtheorem{cor}[thm]{Corollary}

\newcommand{\M}{\mathcal{M}}
\newcommand{\N}{\mathbb{N}}

\newcommand{\R}{\mathbb{R}}

\renewcommand{\SS}{\mathbb{S}}

\renewcommand{\=}{\asymp}

\newcommand{\BF}{\bar{\mathbf{F}}}

\newcommand{\bF}{\mathbf{F}}

\newcommand{\tbF}{\widetilde{\mathbf{F}}}

\newcommand{\tH}{\widetilde{H}}
\newcommand{\BH}{\bar{H}}

\newcommand{\tK}{\widetilde{K}}
\newcommand{\BM}{\bar{M}}
\newcommand{\tM}{\widetilde{M}}

\newcommand{\hM}{\widehat{M}}

\newcommand{\BMs}{\bar{M}\mbox{}^*}

\newcommand{\BT}{\bar{T}}
\newcommand{\hT}{\widehat{T}}

\newcommand{\tn}{\widetilde{\mathbf{\nu}}}

\newcommand{\BW}{\bar{\Omega}}

\newcommand{\hS}{\widehat{\Sigma}}

\renewcommand{\a}{\alpha}

\renewcommand{\d}{\delta}
\newcommand{\e}{\varepsilon}
\newcommand{\f}{\varphi}
\newcommand{\g}{\gamma}

\renewcommand{\l}{\lambda}
\newcommand{\m}{\mu}
\newcommand{\n}{\mathbf{\nu}}

\renewcommand{\t}{\tau}

\newcommand{\W}{\Omega}

\newcommand{\dtau}{\: d \tau}

\newcommand{\bd}{\partial}

\newcommand{\DD}{\mathbb{D}}
\newcommand{\tDD}{\widetilde{\mathbb{D}}}

\newcommand{\none}{\varnothing}
\newcommand{\wo}{\setminus}
\newcommand{\sub}{\subset}

\renewcommand{\bar}{\overline}

\newcommand{\bsm}{\left( \begin{smallmatrix}}
\newcommand{\esm}{\end{smallmatrix} \right)}
\newcommand{\bsd}{\left| \begin{smallmatrix}}
\newcommand{\esd}{\end{smallmatrix} \right|}

\newcommand{\lp}{\left(}
\newcommand{\rp}{\right)}

\newcommand{\dsp}{\displaystyle}
\newcommand{\be}{\begin{enumerate}}
\newcommand{\ee}{\end{enumerate}}
\newcommand{\ba}{\begin{align*}}
\newcommand{\ea}{\end{align*}}
\newcommand{\bi}{\begin{itemize}}
\newcommand{\ei}{\end{itemize}}

\newcommand{\rb}{%
	\begin{tikzpicture}[scale=.07]%
	\newcommand\lv{2.2}%
	\newcommand\yv{.3}%
	\draw (-\lv,\yv) -- (0,\yv);%
	\draw (-\lv,-\yv) -- (0,-\yv);%
	\draw[fill=white] (0,0) circle (1);%
	\draw[white,fill=white] (-\lv,-\yv+.1) rectangle (0,\yv-.1);%
\end{tikzpicture}%
}
\newcommand{\lb}{\reflectbox{\rb}}

\title{\vspace{-1 in}Blow-up Continuity for Type-I, Mean-Convex Mean Curvature Flow}
\author{Kevin Sonnanburg}
\date{}

\begin{document}
\maketitle

\vspace{-.5 in}
\begin{abstract}
	Under mean curvature flow, a closed, embedded hypersurface $M(t)$ becomes singular in finite time.
	For certain classes of mean-convex mean curvature flows, we show the continuity of the first singular time $T$ and the limit set ``$M(T)$'', with respect to initial data. 

We employ an Angenent-like neckpinching argument to force singularities in nearby flows. 
However, since we cannot prescribe initial data, we combine Andrews' $\a$-non-collapsed condition and Colding and Minicozzi's uniqueness of tangent flows to place appropriately sized spheres in the region inside the hypersurface. 
\end{abstract}

\setcounter{section}{-1}
\section{Introduction}

 We study the solution $M(t)$ to mean curvature flow, with initial data $M(0) = M_0$,  near the first singularity at time $T$.
Let $\bF:\mathcal{M}\times [0,T) \to \R^{N+1}$ be a family of smooth embeddings $\mathbf{F}(\cdot, t) = M(t)$, where $\mathcal{M}$ is a closed $N$-dimensional manifold. 
	We say that $M=\left\{ M(t) \right\}_{t \in [0,T)}$ is a mean curvature flow if
	\begin{equation}
		\bd_t \bF = - H \n ,
	\end{equation}
where $H$ is the scalar mean curvature, $\n$ is the \emph{outward} unit normal, and $-H \nu$ is the mean curvature vector. \\
	
We show the continuity of first singular time for two classes of flows. 
As a corollary, we show continuity of the limit set at time $T$. 

\subsection{Main Results}
\label{sec:results}

\begin{thm} \thlabel{thm:T}
	Let $\BM_0 \sub \R^{N+1}$ be a smoothly embedded, closed, mean-convex hypersurface. 
	Let $\BM(t)$ be the solution to mean curvature flow with $\BM(0) = \BM_0$. 
	Assume that $\BM(t)$ shrinks to a point at time $\BT$. 

	For every $n \in \N$, let $M_{n0} \sub \R^{N+1}$ be a smoothly embedded, closed hypersurface that can be expressed as the graph of some function $f_n$ over $\BM_0$. 
Finally, say $\BT$ and $T_n$ are the first singular times for $\BM$ and $M_n$, respectively. 

If $M_{n0} \to \BM_0$ (i.e. $\|f_n\|_{C^2(\BM_0)} \to 0$), then $T_n \to \BT$. 
\end{thm}
Convergence is smooth in the sense of compact graphs over the hypersurface (see \textbf{Closeness} in \S \ref{sec:def}).
\\

The proof of \thref{thm:T} is short and relies heavily on the inclusion monotonicity of mean-convex flow. 
The technique is useful in the proof of our main result \thref{thm:mc}, so we do the proof of \thref{thm:T} in \S \ref{sec:T}, as soon as we have established general notation and definitions. 
The proof of \thref{thm:mc} requires $N=2$, so it is not a strict generalization of \thref{thm:T}. 
\\

Why the restriction to surfaces? 
The technique used in proving \thref{thm:mc} is inspired by the neck-pinching strategy employed by Angenent in~\cite{ang} (see \S\ref{sec:main} for an overview). 
If $\BM_0$ is a surface that does not collapse to a point under the flow, we can use current theory to predict the appropriate neck structure (i.e. the portion of $\BM(t))$ near a singularity is close to a truncated cylinder, say $\SS^1 \times [-a,a]$). 
In higher dimensions, other structures are possible (i.e. the portion of $\BM(t)$ is close to a generalized cylinder that splits off a hyperplane, rather than a line) allowing for too many degrees of freedom in the motion of $\BM(t)$ and nearby flows. 

\begin{thm} \thlabel{thm:mc}
	Let $\BM_0 \sub \R^3$ be a smoothly embedded, closed, mean-convex surface.
	Let $\BM(t)$ be the solution to mean curvature flow with $\BM(0)=\BM_0$.

	For every $n \in \N$, let $M_{n0} \sub \R^3$ be a smoothly embedded, closed surface that can be expressed as the graph of some function $f_n$ over $\BM_0$. 
	Finally, say that $\BT$ and $T_n$ are the first singular times for $\BM$ and $M_n$, respectively. 

	If $M_{n0} \to \BM_0$ (i.e. $\|f_n\|_{C^2(\BM_0)} \to 0$), and $\BM$ is a type-I flow, then $T_n \to \BT$. 
\end{thm}


A corollary to continuity of first singular time is continuity of the limit set. 
However, the notion of a limit set must be made precise. 
For that, we have the following lemma. 

\begin{lem}[Proposition 2.2.6 of~\cite{mant}]
	\thlabel{lem:S}
	Let $M(t) \sub \R^{N+1}$ be a compact mean curvature flow defined for times $t \in [0,T)$, where $T$ is the first singular time. 
	Define $M^*$ to be the set of points $x \in \R^{N+1}$ such that there exists a sequence of times $t_i \nearrow T$ and a sequence of points $x_i \in M(t_i)$, where $x_i  \to x$.

	Then $M^*$ is compact. 
	Furthermore, $x \in M^*$ if and only if for every $t \in [0,T)$, the closed ball of radius $\sqrt{2N(T-t)}$ and center $x$ intersects $M(t)$. 
\end{lem}

Not only does this ensure the existence of a limit set $M^*$ at time $T$, but the uniform convergence allows us to show continuity of $M^*$ in the Hausdorff distance.
\\

As a nice demonstration of the utility of the continuity of first singular time, we give the following corollary. 

\begin{cor} \thlabel{cor:profile}
	Let $\BM_0$ and the sequence $\{M_{n0}\}_n$ be as in \thref{thm:T} or \thref{thm:mc}.
	Then $M^*_n \to \BMs$ in the Hausdorff metric. 
\end{cor}

The proof of blow-up-time continuity involves multiple cases, some of which are complex, so we include an outline of the argument below.

\subsection{Idea of Main Proof} \label{sec:main}

In showing continuity of blow-up time, it is a standard application of well-posedness to conclude that $\dsp \liminf_{n \to \infty} T_n \ge \BT$. 
Thus, for Theorems \ref{thm:T} and \ref{thm:mc}, it is sufficient to show that $T_n \le \BT + \e$ for large $n$. 
Assume in the following that $n$ is large.
\\

Due to mean-convexity, the flow, and nearby flows, move inward.
A hypersurface beginning inside $\BM_0$ will remain inside by a comparison principle. 
A small adjustment in the time parameter slides nearby flows inside $\BM_0$.
Thus, in the following, we can assume $M_{n0}$ is inside $\BM_0$, affording us more control over when $M_n(t)$ becomes singular. 

\paragraph{Proving \thref{thm:T}}
Once we reduce to the case where $M_{n0}$ is inside $\BM_0$, the proof is nearly trivial.
Since $\BM(t)$ shrinks to a point, there is no escape for a hypersurface inside it. 
Either $M_n(t)$ becomes singular before time $\BT$, or it shrinks to a point at time $\BT$. 
\\

That's it for \thref{thm:T}. 
The rest of the subsection describes the proof of \thref{thm:mc}. 

\paragraph{Proving \thref{thm:mc}}
If $\BM(t)$ does not shrink to a point, we show in \S\ref{sec:prelim} that $\BM(t)$ must develop cylindrical singularities. 
In \S\ref{sec:anatomy} we show cylindrical singularities correspond to a structure with a ``neck'' and two ``bulbs''. 
The surface $M_n(t)$ could slip through the neck and survive in just one bulb of $\BM(t)$, so we cannot count on $M_n(t)$ becoming singular just because $M_{n0}$ is inside $\BM_0$. 
Thus more work is required, but we can use well-posedness to force $M_n(t)$ to have a neck structure like $\BM(t)$.  

\subparagraph{Nonsimply Connected Case}
It turns out the case when $\BM_0$ is not simply connected is easier than when it is simply connected. 
Intuitively, any tube-like portion of $\BM(t)$ will enclose a tube-like portion of $M_n(t)$, and the handle structure prevents $M_n(t)$ from wriggling away. 
Practically, we choose a nontrivial loop in $M_n(t)$ and a loop around the neck to create a Hopf link preserved by the flow.
(See Figure~\ref{fig:link})

\subparagraph{Simply Connected Case}
The strategy for the simply connected case is more intuitive, but far more technical. 
Inspired by Angenent's neck-pinching strategy in~\cite{ang}, in each bulb we place a sphere to hold it open while the neck pinches. 
Because mean curvature flow is well-posed, we can choose $n$ large enough that $M_n$ also has two bulbs held open by the spheres 
(see Figure~\ref{fig:ball}).
Angenent forces the neck to pinch by shrinking a donut around it (the Angenent donut). 
Since we do not prescribe initial data it is not clear an appropriate donut exists.
However, the singular time $T$ is given, so we have no need for the donut. 
\\

\begin{figure}[h]
	\centering
\includegraphics[scale=.3]{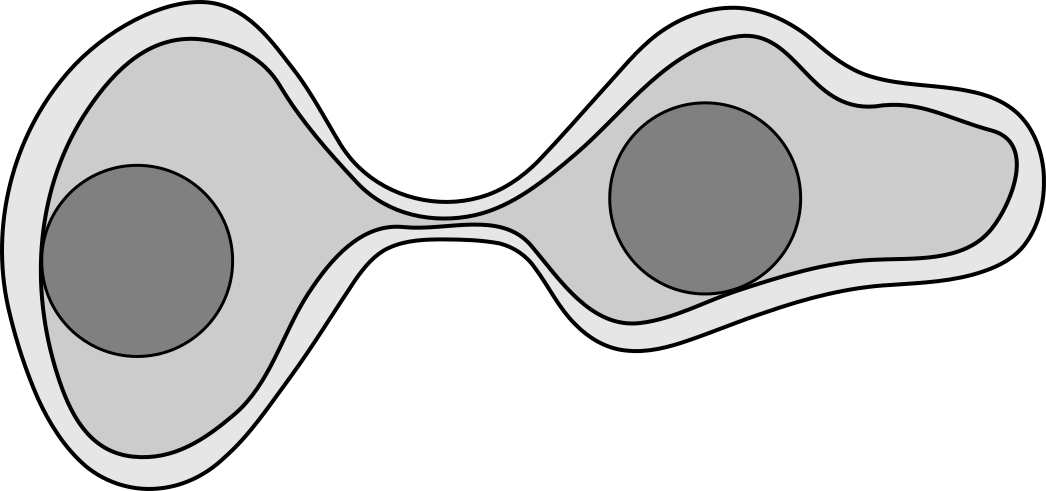}
\caption{Balls inside $\W_n$ with diameter much larger than that of the neck.}
\label{fig:ball}
\end{figure}

On the other hand, not prescribing initial data means we have no a priori knowledge of the appropriate choice of spheres. 
The spheres we choose must fit inside the bulbs \emph{and} survive past the neck pinching. 
Thus we employ a somewhat recent tool of Andrews in~\cite{and}. 
Stated roughly it says that, given $\a >0$ and some flow $M$, if at each point $x$ of $M(t)$ a sphere of radius $r = \frac{\a}{H}$ can fit inside $M(t)$ tangent at $x$, then this condition is preserved by the flow for the same $\a$. 
(See \textbf{Non-Collapsing Condition} and Figure~\ref{fig:andrews} in \S\ref{sec:def})
This allows us to place the spheres, as in the Angenent strategy. 
This is another reason we need mean-convexity. 
\\


When placing the spheres, we must choose their radius $r$, the time $t_0$ at which to place them, and $n$ large enough that we can place them inside $M_n(t_0)$. 
Initially, these quantities appear circularly dependent, so we must find conditions under which one quantity can be chosen independently of the other two. 
The choice of $n$ must depend on $t_0$ because of how we use well-posedness. 
The choice of $t_0$ must depend on $r$ because the neck must be small compared to the spheres so that it pinches before the spherse collapse. 
So we have to choose $r>0$ independent of $t_0$ and $n$. 
Since $r$ is inversely proportional to the curvature, we must show the existence of points in $\M$ for which $H\left( \bF(p,t) \right)$ stays bounded. 
It is sufficient to show there is a regular (nonsingular) point in $\BM^*$. 

\paragraph{Finding a Regular Point}
To the author's knowledge, no result exists to guarantee there is a regular point in the limit set $\BM^*$, so most of \S\ref{sec:anatomy} is dedicated to finding one. 
Under the simply connected assumption, this begins by showing the existence of the desired neck structure with bulbs. 
The type-I assumption restricts the velocity at each point to prevent a bulb from collapsing into the singular point. 
Finally, we use some properties of the singular set to show that it cannot take up a whole bulb at time $\BT$. 
Thus the limit of each bulb must have at least one regular point. 
\\

We conclude the paper by showing continuity of the limit set in \S\ref{sec:profile}, demonstrating a direct use of blow-up-time continuity. 



\section{Preliminaries} \label{sec:prelim}

Below is a list of notation collected for later reference. 

\subsection{Notation}

Throughout this paper, we recycle the use of the following letters.

\begin{itemize}
	\item $\M$: background manifold
	\item 	 $M$: family of hypersurfaces flowing by mean curvature \\
		(i.e. $M=\{M(t)\}_{t \in [0,T)}$)
	\item $\bF$: parameterization $\bF: \M \times [0,T) \to M$
		\item $\Sigma$: fixed hypersurface
	\item  $\W(t)$: open region enclosed by $M(t)$ \\ 
		(definable for closed, embedded hypersurfaces) 
	\item $\nu$: unit \emph{outward} normal vector
	\item  $T$: first singular time 
	\item $A$: second fundamental form
	\item  $H$: mean curvature 


	\item $p$: point in $\M$
	\item $x,y$: point in $\R^{N+1}$
	\item $d_H$: Hausdorff distance
\end{itemize}

The following diacritics are applied to any of the above to associate them to a specific hypersurface. 
\begin{itemize}
	\item $\bar{\square}$: associated with the base flow $\BM$ \\
		(we use $Cl(\cdot)$ for closure)
	\item $\square_0$: initial data
	\item $\square_n$: associated with $M_n$
	\item $\widetilde{\square}$: associated with the rescaled flow $\tM$ (as in \S \ref{sec:def})
	\item $\widehat{\square}$: associated with some auxiliary flow $\hM$ or hypersurface $\hS$ local to a proof
	\item $\square_{\=}$: associated with the neck of a flow \\ 
		(See \thref{defn:sets})
	\item $\square_{\rb / \lb}$: associated with a bulb of the flow \\
		(See \thref{defn:bulb})
	\item $p^* := \lim_{t \to T} \bF(p,t) \in \R^{N+1}$ \\
		(exists due to \thref{lem:TI})
	\item $M^*$: limit set consisting of all such points $p^*$ 
	\item $\square^*$: associated with $M^*$
\end{itemize}
\begin{rk}
	The definition of $M^*$ given here is equivalent to that in \thref{lem:S} due to the compactness of $\M$ and the continuity of the map $p \mapsto p^*$ by Lemma 2.6 of~\cite{stn}.
\end{rk}
The following are used in the setting of the rescaled flow. 
\begin{itemize}
	\item $\l(t) = \left( 2(T-t) \right)^{-\frac{1}{2}}$ is the scaling factor
	\item $\xi = \l x$ is the new spatial variable
	\item $s=-\frac{1}{2} \log(T-t)$ the rescaled time
	\item $s_0 = -\frac{1}{2} \log T$ is the initial time for $\tM$
\end{itemize}
We will also use $\cong$ for homeomorphicity.  \\
\subsection{Definitions} \label{sec:def}

\paragraph{Closeness}
When we say $\Sigma_n \to \Sigma$, we mean smooth convergence in the graph sense.  
That is, let $\Sigma$ and $\Sigma_n$ be closed hypersurfaces.
Assume there is a $C^{\infty}$ function $f_n:\Sigma \to \R$ so that the map $\f_n(x) = x +f_n(x) \nu$ is a smooth diffeomorphism from $\Sigma$ to $\Sigma_n$. 
Then, for every $k \in \N$, $\|f_n\|_{C^k} \xrightarrow[n \to \infty]{} 0$.

\begin{rk}
	Since $M_{n0} \cong \BM_0 \cong \M$, we can consider $\bF_n$ to have $\M$ as its background manifold. 
	This works up to time $T_n$, since the flow preserves embeddings (see \S \ref{sec:tools}). \\

If $d_H$ is the Hausdorff distance, then $\|f\|_{C^0} \le d_H$ (see \thref{lem:fth}).
This is important in the proofs of all three of our results in the previous section. 
\end{rk}

\paragraph{Singular Point}
We say $x \in \R^{N+1}$ is a singular point if there is a sequence $(p_i,t_i) \in \mathcal{M} \times [0,T)$ such that $\bF(p_i,t_i) \to x$ and $|A(p_i,t_i)| \to \infty$ as $i \to \infty$. 

		All mentions of singularities are at the first singular time $T$. 
		\paragraph{Type-I Singularities}
		We say $M$ is a type-I flow if for some $C>0$,
		\[
			\max_{M(t)} |A(p,t)| \le C \left( 2(T-t) \right)^{-\frac{1}{2}} 
			\mbox{ for } t \in [0,T),
		\]
		
		Of course, since $|H| \le \sqrt{N} |A|$, we can say the same of $H$, for a different $C$. 

		\paragraph{Cylinders}
		Since much of this work is set specifically in $\R^3$, it is convenient to distinguish between cylinders and their generalizations. 

		By \emph{generalized cylinder} we mean any set  $(\sqrt{m} \: \SS^m) \times \R^{N-m}$ with $1 \le m \le N-1$ (up to isometry). 

		By \emph{cylinder}, we mean specifically $\SS^1 \times \R$, (up to isometry).

		(Note, the radii are fixed.)
\paragraph{Rescaled Flow}
To understand the asymptotic behavior near the singularity, we consider the rescaled flow: \\
If $x \in \R^{N+1}$ is a singular point of $M$,
	\begin{align}
 \label{eqn:rescaling}
	&	\tbF_x(p,s) := \l(t) \left(\bF(p,t) - x \right), \\
	&	\mbox{where } \l(t)=(2(T-t))^{-\frac{1}{2}} \mbox{ and } s=-\frac{1}{2}\log(T-t) \nonumber ,
	\end{align}
	for $s \in [s_0,\infty)$, where $s_0 = -\frac{1}{2} \log(T)$ and we use $\xi = \l x$ as the spatial variable. 
We will refer to this rescaling as ``the rescaled flow''.

As introduced in \cite{hui}, $\tbF_x$ solves
\begin{equation} \label{eqn:rescale}
		 \bd_s \tbF_x = \tbF_x - \tH_x \tn_x,
	\end{equation}
	where $\tH_x = \tH\left( \tbF_x(p,s) \right)$ and $\tn_x = \tn\left( \tbF_x(p,s) \right)$. 
Objects associated with the rescaled flow are indicated by a tilde.
For simplicity, we will mostly be dealing with the flow rescaled around the origin, meaning $x=0$.
In that case we will \emph{omit} the subscript: $\tbF:=\tbF_0$.\\

\paragraph{Tangent Flow}

Let $M$ be a mean curvature flow for times $t \in [0,T)$.
	Fix $(x_0,t_0) \in \R^{N+1} \times \R^+$.
	Then one can check that the rescaling
	\[
		M_{\mu,(x_0,t_0)}(t) := \mu \left( M\left(t_0- \mu^{-2} (-t) \right) -x_0 \right) \mbox{  for  } t \in \left[-\mu^2 t_0, \mu^2(T-t_0) \right)
	\]
	is also a solution to mean curvature flow.
	Taking a sequence $\mu_i \nearrow \infty$, consider the sequence of rescalings $M^i_{(x_0,t_0)}(t):= M_{\mu_i,(x_0,t_0)}(t)$.
	If $M_{(x_0,t_0)}^i(t)$ has a subsequence converging in $i$ to a flow $M^{\infty}(t)$, that limit is called a \emph{tangent flow} (we discuss the manner of convergence in \S \ref{sec:TI}). 
	\\

	It will be most convenient to consider tangent flows at $(x_0,t_0)=(0,T)$, so in that case we will \emph{omit} the subscript: 
	\begin{equation}
		\label{eqn:tgt} 
		M^i = M^i_{(0,T)} 
		= \mu_i M(T+\mu_i^{-2} t) 
		\mbox{ for } t \in [-\mu_i^2 T,0) .
	\end{equation}

	\subsection{Tools}
	\label{sec:tools}
The following are previously established results, and will be taken for granted throughout this work. 
\paragraph{Well-posedness} (Theorem 1.5.1 of \cite{mant})

Given $M_0$ is compact and immersed (we require it to be embedded anyway),
well-posedness for mean curvature flow has been established in multiple contexts, but the classical case is nicely laid out in \S 1.5 of~\cite{mant}, along with the PDE background in Appendix A of the same work. \\

\begin{thm}
	\thlabel{thm:wp}
	For any initial, smooth, compact hypersurface in $\R^{N+1}$ given by an immersion $\bF_0:\M \to \R^{N+1}$, there exists a unique, smooth solution to mean curvature flow $\bF: \M \times [0,T)$ for some $T>0$, with $M_0=\bF_0(\M)$.

		Moreover, the solution depends smoothly on the initial immersion $\bF_0$.
\end{thm}

The last statement is meant in the sense of \textbf{Closeness} in \S\ref{sec:def} above. 
That is, if $M_{n0}$ is the graph of a smooth function $f_{n0}$ over $\BM_0$, then there is a $\d>0$ such that for each $t \in [0,\d)$, $M_n(t)$ is a graph over $\BM(t)$. 
Furthermore, for any nonnegative integer $k$, if $\|f_{n0}\|_{C^k(\BM_0)} \to 0$, then $\|f_n(t)\|_{C^k(\BM(t))} \to 0$ for each fixed $t \in [0,\d)$. 
Since conditions at each time can be thought of as new initial data, this process can be repeated for larger and larger $n$, so that we have the same convergence at later nonsingular times and 
	\[
		\liminf_{n \to \infty} T_n \ge \BT .
	\]
	Therefore, in showing continuity of first singular time, we need only show that
	$\dsp \limsup_{n \to \infty} T_n \le \BT$. 
		\\

In particular, $k=2$ gets us uniform convergence of $H_n(t)$ to $\BH(t)$.
Precisely, this means that if $M_n(t)$ is a graph of $f_n$ over $\BM(t)$, so that for $x \in \BM(t)$ there is $y_n = x+f_n(x) \nu$, then $H_n(y_n) \rightarrow \BH(x)$. 
The convergence is uniform since $\BM(t)$ is compact. 
Due to \thref{lem:fth}, $k=0$ also gets us convergence of $M_n(t)$ to $\BM(t)$ in the Hausdorff distance. 
\\

	\paragraph{Embedding Preservation} (Theorem 2.2.7 of~\cite{mant})

If the initial hypersurface is compact and embedded, then it remains embedded during the flow.
 
In particular, for any $t_1,t_2 \in [0,T)$, $M(t_1) \cong \M \cong M(t_2)$.
	For example, a simply connected surface would stay simply connected for the duration of the flow. 

	\paragraph{Minimum Principle} (Proposition 2.4.1 of~\cite{mant})

		Mean-convexity is preserved by mean curvature flow.
		In fact, $H$ immediately becomes positive everywhere, and $\min_{M(t)}H$ is nondecreasing in $t$.
		This means that $H$ is strictly positive for $t \in (0,T)$. 
		
\paragraph{Comparison Principle} (Corollary 2.2.3 of~\cite{mant})

Similar to comparison principles for other parabolic equations, initially disjoint solutions remain disjoint. 
More specifically, let $M_1$ and $M_2$ be compact, embedded mean curvature flows, with respective first singular times $T_1$ and $T_2$. 
Assume $M_2$ begins strictly inside $M_1$. 
Then that containment is preserved until time $\min\left\{ T_1,T_2 \right\}$. 
\\

Due to the subsequent discussion in~\cite{mant} (right before Corollary 2.2.5), this can be extended to allow the hypersurfaces to touch initially. 
If they are not initially identical hypersurfaces, they will immediately be disjoint after any short amount of time. 
Then Corollary 2.2.3 of~\cite{mant} can again be applied. 

\paragraph{Non-Collapsing Condition} (Theorem 3 of~\cite{and})

From Definition 1 of~\cite{and}: We say a mean-convex hypersurface $M$ bounding an open region $\W$ in $\R^{N+1}$ is $\alpha$-non-collapsed if, for every $x \in M$, there exists a sphere of radius $\frac{\alpha}{H(x)}$ contained in $Cl(\W)$ with $x \in \partial \Omega$ (See Figure~\ref{fig:andrews}). 
We have that the condition is preserved, with the same $\a$, by mean curvature flow up to the first singular time.  
\\
\begin{figure}[h]
	\centering
\includegraphics[scale=.2]{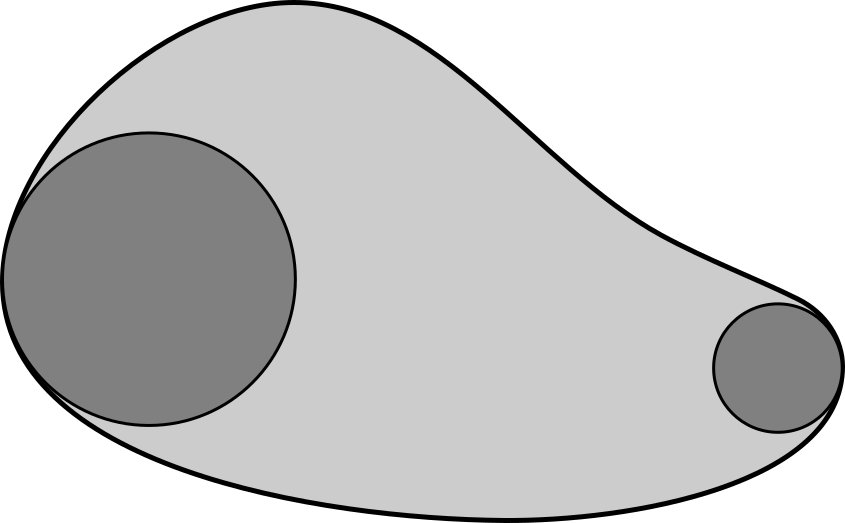}
\caption{Balls of radii varying with curvature.}
\label{fig:andrews}
\end{figure}

\paragraph{Hopf Link}
The following scenario takes place in $\R^3$ at a fixed time $t_0 \in [t_{\=},T)$, so we \emph{omit} time in the discussion below.  We assume $M$ has a cylindrical singularity (see \thref{defn:sing}) at the origin, with the $x_2$-axis as the axis of the cylinder. 
	(For more details on the notation, see \thref{lem:neck} and \thref{defn:sets}.)
\\

More than once, we use the notion of the Hopf link to ``trap'' part of a curve or surface in a pinching neck to force a singularity, as in Figure~\ref{fig:link}. 
Assume $M$ is not simply connected. 
We will have isolated a ``neck'' of $M$ as the intersection between it and a truncated, filled cylinder $K$. 
Within $K$ is $\DD = K \cap \{x_2=0\}$.
The boundary of $\DD$ is a circle, which is a simple, closed curve.  
Now take another simple, closed curve $\g$ lying in $M$ that passes through $\DD$, transversely, exactly once. 
\\

The two curves $\g$ and $\DD$ form a Hopf link, which is a nontrivial link in $\R^{N+1}$. 
To see this, rotate coordinates so the page is the $x_1x_3$-plane and the $x_2$-axis has its positive direction coming out of the page, and $\DD$ lies in the page, as depicted in Figure~\ref{fig:cross}. 
Consider only $\DD$ and $\g$, but extend $K$ to the infinite cylinder $K'=\{\sqrt{x^2+y^2} < 4 \l\}$, so that for a point in $K'$, $x_2>0$ and $x_2<0$ correspond to being ``in front of'' and ``behind'' $\DD$, respectively (see Figure~\ref{fig:cross}).
Note the inequality is strict so that $K'$ is open, unlike $K$.
That way we need not count the linking numbers below when $\g$ merely touches the boundary of $K'$.
\\

Let us compute the linking number, considering $\g$ to be going clockwise, (that is, the projection of $\g$ into the $x_1x_3$-plane has a positive winding number with respect to the origin).
As illustrated, if $\g$ leaves $K'$ while $x_2>0$ or enters $K'$ while $x_2<0$, we add 1.
Contrastly, we subtract 1 any time $\g$ enters $K'$ while $x_2>0$ or leaves $K'$ while $x_2<0$. 
Since $\g$ intersects $\DD$ at only one point, $x_2$ only changes sign once while $\g$ is in $K'$. 
This means we need only count the first time $\g$ leaves $K'$, and the last time it enters, since all other times at which it enters, it must exit on the same side of the $x_1x_3$-plane, and they cancel. 
The sum is thus two, and the linking number is one. \\

By \thref{lem:neck}, $M \cap \bd \DD = \none$, so $\g \cap \bd \DD = \none$. 
Thus neither curve can contract to a point via homotopy, if the embedding is to be preserved.
The curve $\g$ undergoes homotopy because the flow is continuous. 
The curve $\bd \DD$ shrinks homothetically, by its construction in \thref{defn:sets}, which is also a homotopy. 
Therefore, the link is preserved up to time $T$.

\begin{figure}
	\centering
	\begin{minipage}{.4 \textwidth}
		\centering
		\includegraphics[width=1 \linewidth]{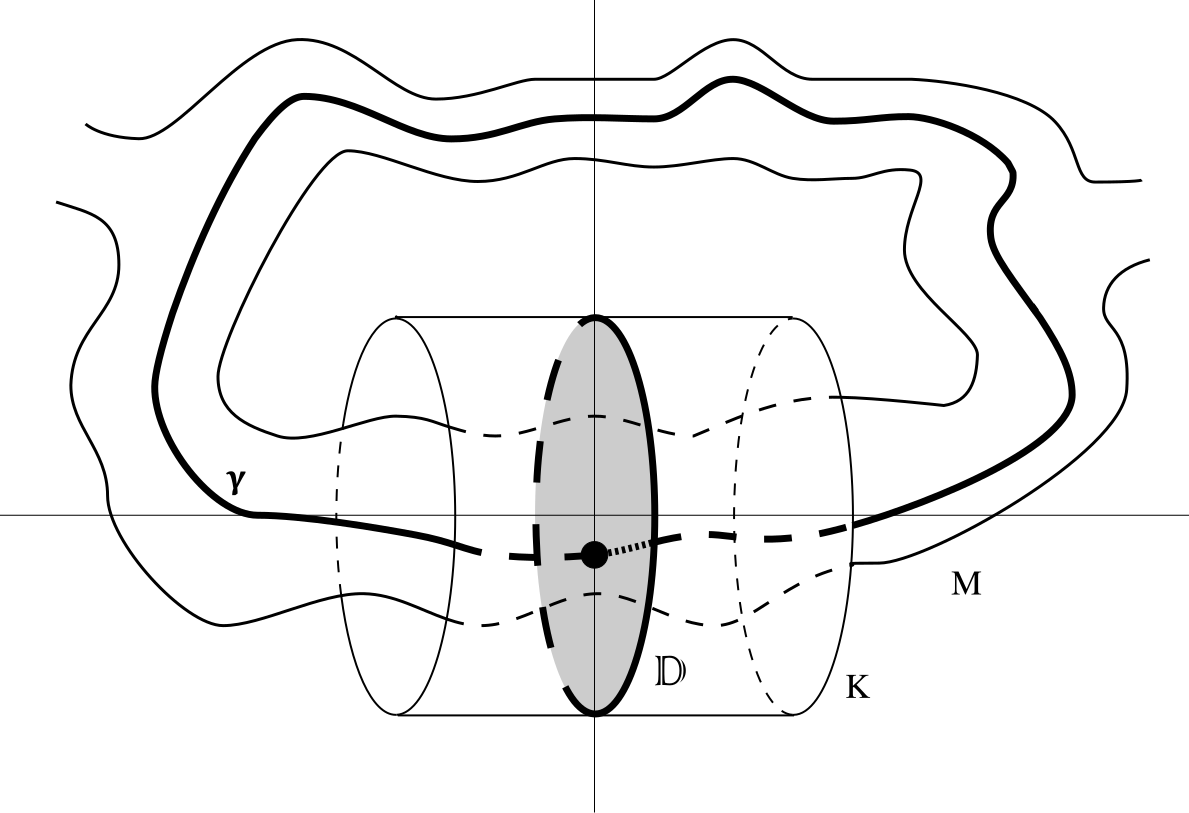}
		\subcaption{Two (thick) curves making a Hopf link}
		\label{fig:link}
	\end{minipage}
	$\longrightarrow $
	\begin{minipage}{.4 \textwidth}
		\centering
		\includegraphics[width=.6 \linewidth]{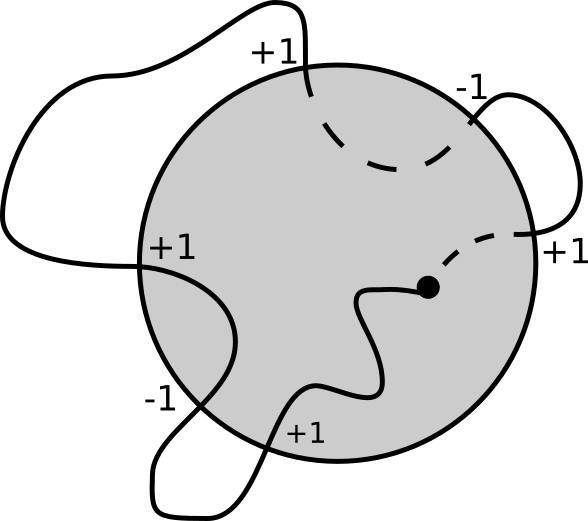}
		\subcaption{Sum: 2, Linking number: 1}
		\label{fig:cross}
	\end{minipage}
	\caption{}
\end{figure}
\subsection{Some Background on Blow-ups}
\label{sec:TI}

In \S \ref{sec:def} we introduced two common blow-up techniques, the rescaled flow and tangent flows, and we use both. 
For our purposes, there are stronger results for tangent flows, such as uniqueness against subsequence.
On the other hand, the rescaled flow is convenient since it only deals with one flow, and certain objects remain stationary. 
Luckily, the calculation in (\ref{eqn:ttr}) below shows that any sequence of times $s_i \nearrow \infty$ in the rescaled flow corresponds to a particular tangent flow. \\

White showed in~\cite{whimc} that, in the mean-convex case, all tangent flows are either planes, (generalized) cylinders, or spheres. 
Furthermore, we have from~\cite{sw} that planes are ruled out for times approaching the first singular time from below. 
Our main result is for $N=2$, so we deal mostly with cylinders as blow-up limits. 
These limits, however, are those of subsequences. 
The question arises whether the limit depends on the subsequence.
That is, the cylinder shape and radius are fixed, but can the orientation change per subsequence?
Colding and Minicozzi find in~\cite{cmu} that it cannot:
If one tangent flow is a cylinder, then they all are. In fact, they are all the same cylinder. \\

What about the rescaled flow? 
Huisken showed in~\cite{hui} that the rescaled flow, when centered around a singular point, converges smoothly on compact subsets to a stationary limit. 
This corresponds to a self-similar flow in the nonrescaled setting, which we know from~\cite{cm} indicates a cylinder.
So if we can connect this notion of rescaling to tangent flows, we can control limits of $\tM(s)$, because of the uniqueness of cylindrical tangent flows. 

For our tangent flow, let $(x_0,t_0)=(0,T)$.
Now choose some sequence $t_i \nearrow T$, with corresponding $s_i \nearrow \infty$.
Recalling $\l(t) = \left( 2(T-t) \right)^{-\frac{1}{2}}$, let $\mu_i=\l(t_i)$.
Then $\mu_i \nearrow \infty$. 
Then, using the notation from the \emph{rescaled flow} and \emph{tangent flow} definitions from \S \ref{sec:def}, we make the observation that, by (\ref{eqn:tgt})
 \begin{align}
	 \label{eqn:ttr}
	 M^i\left( -\frac{1}{2} \right) 
	 &= \mu_i M\left( T-\m_i^{-2}\left( \frac{1}{2} \right) \right) \nonumber \\
	 &= \l(t_i) M\left(T-\frac{1}{2}\l^{-2}(t_i)\right) 
	 = \l(t_i) M(t_i) 
	 = \tM(s_i) ,
 \end{align}
with the rescaled flow on the right, and the tangent flow rescalings on the left. 
\\

We now make these ideas more precise. 

\begin{lem}
	\thlabel{lem:cyl}
	Let $M$ be a type-I mean curvature flow with a singular point $x \in M^*$ at the first singular time $T$. 
	Assume there is at least one tangent flow at $(x,T)$ that is a generalized cylinder.
	Then $\dsp \lim_{s \to \infty} \tM(s) = \tM_{\infty}$ exists and is the same generalized cylinder. 
\end{lem}

Convergence is smooth on compact subsets of $\R^3$. 
That is, for large $s$, $\tM(s)$ can be locally described as a graph, over $\tM_{\infty}$, of some function $u$, which is $C^k$-small. 
(See \textbf{Closeness} in \S \ref{sec:def}.)
We only need $k=2$, and will write ``$\xrightarrow{C_c^2}$'' for this type of convergence. 

\begin{proof}
	Assume, without loss of generality, the the singular point is the origin. 
	Assume there is at least one tangent flow at $(0,T)$ that is generalized-cylindrical. \\

	Let $s_i \nearrow \infty$.
	Since $M$ is type-I, Theorem 3.4 of~\cite{hui} provides a subsequence $\left\{ s_{i_j} \right\}$ and $\tM_{\infty}^{\left\{ s_{i_j} \right\}}$. 
	Now let $\mu_i = \l(t_i)$, so that by (\ref{eqn:ttr}),
	\[
		M^{i_j}\left( -\frac{1}{2} \right) 
		= \tM(s_{i_j}) \xrightarrow[j]{C_c^2} \tM_{\infty}^{\left\{ s_{i_j} \right\}} .
	\]
	That means $\tM_{\infty}^{ \left\{ s_{i_j} \right\}}$ is a tangent flow. 
	Since there is a generalized-cylindrical tangent flow by hypothesis, Theorem 0.2 of~\cite{cmu} says all tangent flows at $(0,T)$ are the very same generalized cylinder $\tM_{\infty}^{ \left\{ s_{i_j} \right\}}$. \\

	That is, we now have that every sequence $s_i \nearrow \infty$ has a subsequence $s_{i_j}$ for which $\tM_{\infty}^{ \left\{s_{i_j}\right\} }$ is the same generalized cylinder as above. 
	Then this must be true of all sequences of times $s_i \nearrow \infty$. 
	Finally, this means that $\dsp \lim_{s \to \infty} \tM(s)$ makes sense and is a unique generalized cylinder $\tM_{\infty} = \tM_{\infty}^{ \left\{ s_{i_j} \right\}}$. 
\end{proof}

\begin{lem}
	\thlabel{lem:sph}
	Let $M$ be a type-I mean curvature flow with a singular point $x \in M^*$ at the first singular time $T$. 
	Assume there is at least one tangent flow at $(x,T)$ that is a sphere.

	Then $\dsp \lim_{s \to \infty} \tM(s) = \tM_{\infty}$ exists and is the same sphere. 
\end{lem}

\begin{proof}
	Assume, without loss of generality, that the singular point is the origin. 
	Assume there is some tangent flow at $(0,T)$ that is a sphere.  \\

	Then there is some sequence $\mu_i \nearrow \infty$ (assume $\mu_i >2$) for which each $M_{\mu_i}(t)$ is defined on $[-T,0)$  whose limit flow is a sphere. 
		More precisely,
		recall that $\l(t) = \left( 2(T-t) \right)^{-\frac{1}{2}}$ and $s(t) = -\frac{1}{2} \log (T-t)$ so $M(t) = \l^{-1}(t) \tM(s)$. 
		Let
		\[
			s_i = \log \mu_i - \frac{1}{2} \log(-t) \nearrow \infty .
		\]
		Then by (\ref{eqn:tgt}),
		\begin{align*}
			M_{\mu_i}(t) 
			= \mu_i M(T+\mu_i^{-2}t) 
			&= \mu_i \l^{-1}(T+\mu_i^{-2}t) M\left(\log \mu_i - \frac{1}{2} \log(-t) \right) \\
			& = \sqrt{-2t} \tM(s_i)
		\end{align*}

		Since $M$ is type-I, Theorem 3.4 of~\cite{hui} provides a subsequence $s_{i_j}$ so that $\tM(s_{i_j})$ converges to some $\tM_{\infty}^{ \left\{s_{i_j} \right\}}$ in $C_c^2$. 
		By hypothesis, that limit is a sphere. 
		Given the $C_c^2$ convergence, that means there is some large $s_1$ for which $\tM(s_1)$ is strictly convex. 
		Finally, we know from Huisken's main theorem in~\cite{huisph} that $\tM(s)$ remains convex after $s_1$ and converges to a sphere in $C^2$ as $s \nearrow \infty$. 
		That means $\tM_{\infty}$ makes sense and is a sphere. 
\end{proof}

\begin{defn}
	\thlabel{defn:sing}
	Assume $M$ has a singular point $x$ and that $\tM_{\infty}^x$ exists. 
	We call $x$ a spherical ((generalized) cylindrical) point if $\tM_{\infty}^x$ is a sphere ((generalized) cylinder).
	Here the spheres and (generalized) cylinders have the radii specified in \textbf{Cylinders} from~\S\ref{sec:def}. 
\end{defn}

\begin{cor}
	\thlabel{cor:sing}
	Let $M$ be a smoothly embedded, closed, type-I, mean-convex mean curvature flow. 

	Then the flow $M$ has at least one singular point and: 

	Either $M(t)$ becomes convex and shrinks to a point, or all singular points of $M$ at time $T$ are cylindrical. 
\end{cor}

\begin{proof}
	Since $M_0$ is compact, one can place a sphere containing it. 
	By the comparison principle, $M(t)$ must become singular before the sphere collapses. 
	Call the first singular time $T$. 
	\\

	Without loss of generality, let the origin be a singular point of $M$. 
	Take a sequence of rescale times $s_i \nearrow \infty$. 
	Since $M$ is type-I, we again know from Theorem 3.4 of~\cite{hui}, that there is a subsequence $s_{i_j}$ for which $\tM_{\infty}^{ \left\{ s_{i_j} \right\}}$ exists. 
	By (\ref{eqn:ttr}), that is a tangent flow. 
	Therefore, at least one tangent flow at $(0,T)$ exists. 
	\\

	Since $M(t)$ is mean-convex, Theorem 1.1 of~\cite{whimc} says every tangent flow at $(0,T)$ is a plane, a sphere, or a cylinder. 
	However, Corollary 8.1 in~\cite{sw} precludes any planar tangent flows at the first singular time. 
	Thus every tangent flow at $(0,T)$ is a sphere or a cylinder. 
	Then by \thref{lem:cyl} and \thref{lem:sph}, every singular point is either spherical or cylindrical. \\

	Now assume $M$ has a spherical point. 
	Without loss of generality, assume it is the origin. 
	If $\tM_{\infty}$ is a sphere, then there is some $s$ for which $\tM(s)$ is convex. 
	Then the same is true of $M(t)$ at some time $t$. 
	Therefore $M(t)$ collapses to a point, by the main theorem of~\cite{huisph}. 
	Thus, the existence of spherical points and the existence of cylindrical points are mutually exclusive.

%
%

\end{proof}

We conclude with a strengthening of \thref{lem:S}. 
Although the precise statement of \thref{lem:S} is still convenient for proving \thref{cor:profile}, the type-I assumption grants us control over the flow of specific points $\bF(p,t)$. As in Lemma 3.3 in~\cite{hui}:

\begin{lem}
	\thlabel{lem:TI}
	Let $M$ be a type-I mean curvature flow with first singular time $T$. 

	Then for each $p \in \M$ there is a $p^* \in \R^3$ and $C>0$ for which 
	\[
		|\bF(p,t)-p^*| \le C \left( 2(T-t) \right)^{\frac{1}{2}} 
		= C \l^{-1}(t) .
	\]
\end{lem}

\begin{proof}
	Given the type-I bound $H \le C \left( 2(T-t) \right)^{-\frac{1}{2}}$ for some $C>0$,
	\[
		 \left| \int_t^T H\left( \bF(p,\tau) \right) \: d \tau \right|
		\le \int_t^T C \left( 2(T-\tau) \right)^{-\frac{1}{2}} \: d \tau 
		= \left( 2(T-t) \right)^{\frac{1}{2}} .
	\]
	Then for every sequence $t_i \nearrow T$, $\bF(p,t_i)$ is a Cauchy sequence. 
	Thus there is a $\dsp p^* := \lim_{t \to T} \bF(p,t)$ exists.
	Since
		\[
			p^*-\bF(p,t) = \int_t^T \bd_{\tau} \bF(p,\t) \dtau
			= \int_t^T -H\left( \bF(p,\tau)  \right) \nu \dtau ,
		\]
	we are done. 
\end{proof}

This lemma will be useful more than once in \S 2.

\subsection{Proof of \thref{thm:T}}
\label{sec:T}

The following proof uses Hausdorff distance, but our notion of closeness is in terms of $\|f_n\|_{C^k}$. 
We must connect the two.

\begin{lem}
	\thlabel{lem:fth}
	Let $\Sigma$ and $\hS$ be hypersurfaces such that $\hS$ is a graph of a smooth function $f$ over $\Sigma$. 
	Then $d_H(\hS,\Sigma) \le \|f\|_{C^0}$. 
\end{lem}

\begin{proof}
	First of all, since the mapping $y = \f(x) = x+f(x) \nu$ is a homeomorphism, each $x \in \Sigma$ is uniquely paired with a $y \in \hS$, and vice versa. 
	\\

	For a given pair of points $(x,y) \in \Sigma \times \hS$, $|f(x)| = |x-y|$. 
	However, $y$ might not be the closest point on $\hS$ to $x$, so $|f(x)| \ge \inf_{y} |x-y|$.
	Then
	\[
		\|f\|_{C^0} =\sup_x |f(x)| \ge \sup_x \inf_y |y-x| .
	\]

	Now let $g(y) := f\left( \f^{-1}(y) \right)$. 
	Then 
	\[
		\|f\|_{C^0} = \sup_x|f(x)| = \sup_y |g(y)| \ge \sup_y \inf_x |x-y|.
	\]
	$|g(x)| \ge \inf_x |y-x|$. 
	Thus we have
	\[
		\|f\|_{C^0} 
		\ge \max \left\{ \sup_{x \in X} \inf_{y \in Y} |x-y|, \sup_{y \in Y} \inf_{x \in X} |y-x| \right\} 
		= d_H(\Sigma,\hS) .
	\]
\end{proof}

\begin{proof}[Proof of \thref{thm:T}]
	Let $\BM$ and a sequence $\{M_n\}_n$ be as in \thref{thm:T}.  
	By well-posedness, we already have that $\dsp \liminf_{n \to \infty} T_n \ge \BT$. 
	So we need only show that $\dsp \limsup_{n \to \infty} T_n \le \BT$. 
	For an illustration of the following, see Figure \ref{fig:inside}.
	\\

	Let $0<\varepsilon$. 
	Define $\hM_n(t)=M_n(t+\e)$.
	We need $\hT_n > \e$ for the following proof to make sense (since the intervals of existence times for $\BM$ and $\hM_n$ need to overlap).
	However, the goal at the end of the proof is to show that, for large $n$, $\hT_n < \BT+\e$, so if $\hT_n \not>\e$, we are done. 
	So we can just assume $\hT_n > \e$. 
	\\

	Now $\hM_{n0} = M_n(\varepsilon)$. 
	Since $\BM(t)$ strictly is mean-convex for $t=[\frac{\varepsilon}{2},\BT)$, its velocity at every point is inward with positive speed. 
Thus $\hM_0(\e) \sub \BW_0$, and we have the Hausdorff distance $d=d_H \lp \BM(\e),\BM_0 \rp > 0$. 
By well-posedness, there is an $n_0>0$ so that if $n \ge n_0$, then $d_H(\BM(\e),\hM_{n0})<\frac{d}{2}$ (see \thref{lem:fth}).
So assume $n \ge n_0$.
Rearranging
	\[ 
		d_H(\BM(\e),\BM_0) \le  d_H(\BM(\e),\hM_{n0}) + d_H(\hM_{n0},\BM_0) 
	\]
	gets us
	\[
		d_H(\hM_{n0},\BM_0) 
		\ge d_H(\BM(\e),\BM_0) - d_H(\BM(\e),\hM_{n0}) 
		> d-\frac{d}{2} 
		=\frac{d}{2} 
		> 0 .
	\]
Thus $\hM_{n0} \sub \BW_0$. \\

\begin{figure}[h]
	\centering
\includegraphics[scale=.3]{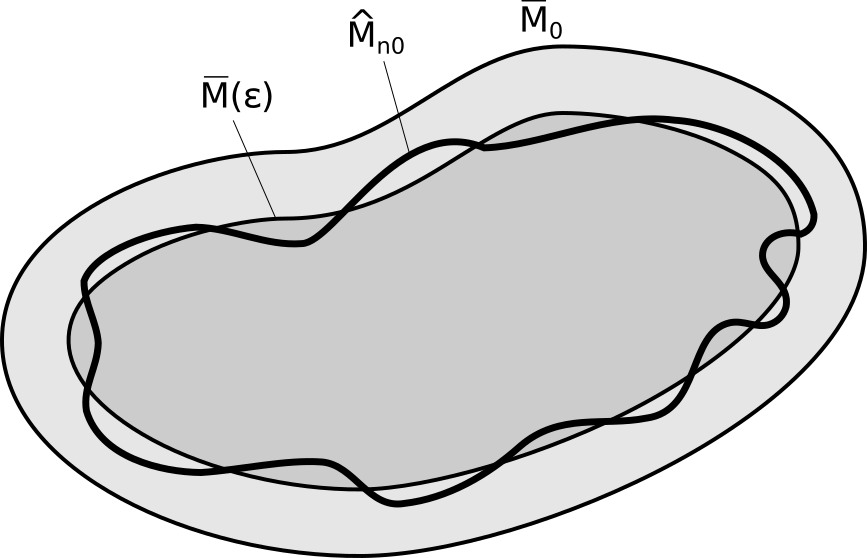}
\caption{$\hM_{n0}$ is closer to $\BM(\e)$ than $\BM_0$, so is contained in $\BW_0$.}
\label{fig:inside}
\end{figure}

Because $\BM(t) \to 0$, the comparison principle tells us $\hM_n(t)$ must become singular no later than $\BM(t)$ does.
So we see that $T_n = \hT_n + \e  \le \BT + \e$.
Since that is true for any $n \ge n_0$, we have $\dsp \limsup_{n \to \infty} T_n \le \BT + \e$.
Since $\e$ was arbitrary, we are done. 

\end{proof}

\section{``Anatomy'' of $M$}
\label{sec:anatomy}

\newcounter{conditions}

\thref{thm:mc} only applies to surfaces. 
We can tell from the proof of \thref{thm:T} that the case when $\tM_{\infty}$ is a sphere is easily resolved, since in this case $M(t)$ would shrink to a point. 
However, we will need much more understanding about the case when $\tM_{\infty}$ is a cylinder. 
The key to all of our analyses is the neck structure that forms at a type-I cylindrical singularity. 
Nearly all the results in this section make heavy use of this structure (detailed in \thref{lem:neck} and \thref{defn:sets}), so we make the following assumptions for \emph{this entire section}. 

\begin{enumerate}[(i)]
	\item $M_0$ (and therefore each $M(t)$) is a smoothly embedded, closed, mean-convex surface. \label{it:first}
	\item $M$ is type-I
	\item $M$ has only cylindrical singularities
	\item The singularity in question is at the origin, and the axis of its cylinder is the $x_2$-axis \label{it:last}
		\setcounter{conditions}{\value{enumi}}
\end{enumerate}
	\subsection{Neck Formation} \label{sec:neck}
	
	We need to describe very precisely what we mean by neck structure.

	\begin{lem} 
		\thlabel{lem:neck}
		Define the solid truncated cylinder
		\[
			\tK = \left\{ \xi: |\xi_2| \le 4 \mbox{ and } \sqrt{\xi_1^2+ \xi_3^2} \le 4 \right\} .
		\]
		 Let $\tn_{\infty}$ to be the outward normal on $\tM_{\infty}$. 
		 (See Figure~\ref{fig:K})
		Then the surface $\tM_{\infty}$ is the unit-radius cylinder whose axis is the $\xi_2$-axis, and the following hold:

		Furthermore for every $0< \e <1$ 
		there is $s_{\=}>0$ (``s-neck'') and smooth $u: (\tM_{\infty} \cap \tK) \times [s_{\=},\infty) \to \R$ such that, if $s>s_{\=}$, then for each $\xi \in \tM_{\infty} \cap \tK$, the mapping 
			\[
				\xi \mapsto \xi+u(\xi,s) \tn_{\infty}
			\]
			is a homeomorphism from $\tM_{\infty} \cap \tK$ to $\tM(s) \cap \tK$, and $\left\|u \right\|_{C^2} < \varepsilon$. 

	\end{lem}
	\begin{rk}
		\thlabel{rk:lid}
		Since $\|u\|_{C^2}<\e< 1$, $\tM$ does not intersect the ``side'' of $\tK$, and must intersect the ``ends'' of $\tK$ transversely.
		Also $\tn_{\infty}$ is parallel to the ``lids'' of $\tK$ ($\tK \cap \left\{ \xi_2 = 4 \right\}$), so there are no questions about the surjectivity of the homeomorphism. 
	\end{rk}

\begin{proof}
	By~\thref{defn:sing}, $\tM_{\infty}$ exists and is a specific cylinder with some radius $r$ to be found. 
	From (\ref{eqn:rescale}) we know 
	\[
		0 = (\bd_s \tbF)^{\perp} = \tbF^{\perp} - \tH \tn 
		= \left( r-\frac{1}{r} \right) \tn ,
	\]
	so we must have $r=1$.
	Theorem 3.4 of~\cite{hui} says that the convergence of $\tM(s)$ to $\tM_{\infty}$ is $C^2$ in compact subsets of $\R^3$ (or to any order one likes), in the sense of \textbf{Closeness}, as in \S \ref{sec:def}. 
\\

	The mapping $\xi \mapsto \xi+u(\xi,s)$ is obviously continuous and invertible by its definition. 
	Since $\tM_{\infty} \cap \tK$ is compact, the mapping is a homeomorphism. 
	Now this lemma is just a particular (consequent, but not equivalent) expression of that convergence. 
\end{proof}

\paragraph{Note:}
For the most part, depictions will follow: 
\begin{center}
	\begin{tabular}[]{rcl}
		$x_1$-axis &:& longitudinal (into the page)\\
		$x_2$-axis &:& lateral \\
		$x_3$-axis &:& vertical
	\end{tabular} 
\end{center}
And we also sometimes use $x_2$ as a coordinate function. 

\begin{defn}[Useful Sets and Quantities]
	\thlabel{defn:sets}
			For later use, also define the disk $\tDD= \{\xi_2=0\} \cap \tK$, orthogonal to the axis of $\tK$. 
			\\

		We write $K$ and $\DD$, without tildes to denote their nonrescaled versions.
		That is $K(t) = \l^{-1}(t)\tK$ and $\DD(t) = \l^{-1}(t) \tDD$. 
			Then we can write $M_{\=}(t) = M(t) \cap K(t)$, and $\W_{\=}(t) = Int(\W(t) \cap K(t))$ (``$M$-neck'' and ``$\W$-neck'').
\\

	Throughout this paper, we refer to $t_{\=}$ (``t-neck''), corresponding to $s_{\=}$.
	After time $t_{\=}$ ($s_{\=}$), we say $M(t)$ ($\tM(s)$) ``has a neck'' for $t \in [t_{\=},T)$ ($s_{\=} \in [s_{\=},\infty)$). 
		Any mention of $t_{\=}$ ($s_{\=}$) hereafter implies the presence of all the structures given in \thref{lem:neck} and \thref{defn:sets}, in $M(t)$ ($\tM(s)$).
			(See Figure~\ref{fig:neck})
\begin{figure}[h]
	\centering
	\includegraphics[scale=.3]{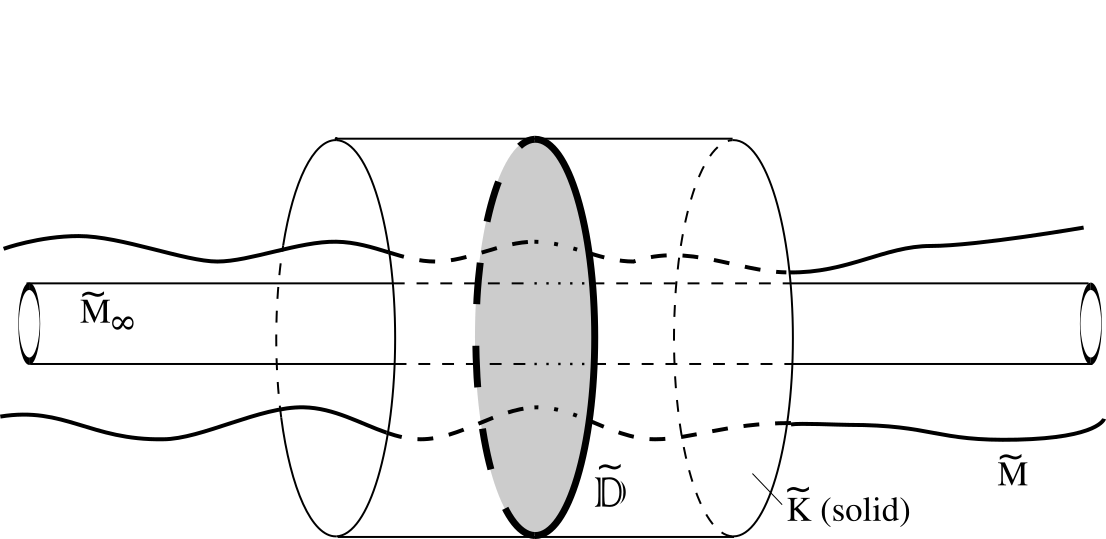}
	\caption{$\tK$ and $\tDD$.}
	\label{fig:K}
\end{figure}
		\end{defn}

		\subsection{Bulb Preservation}

		In the proof of \thref{thm:mc}, it turns out the case where $\BM_0$ is not simply connected is simpler than when it is, due to the the Hopf link construction in \S \ref{sec:tools}. 
	However, in the simply connected case, we need to place spheres inside the bulbs.
	In order to place the spheres as planned, we hope that removing $\W_{\=}(t)$ splits $\W(t)$ into two connected components that always have enough room to fit the desired spheres,
	as in Figures~\ref{fig:neck} and~\ref{fig:ball}.
	So in addition to conditions (\ref{it:first})-(\ref{it:last}), for the \emph{rest of this section} we assume
	\begin{enumerate}[(i)]
			\setcounter{enumi}{\value{conditions}}
		\item \label{it:simply} $M_0$ (and therefore each $M(t)$) is simply connected.
			\setcounter{conditions}{\value{enumi}}
	\end{enumerate}

\begin{lem}
	\thlabel{lem:comp}
	Assume that $M_0$ is simply connected.

	Then for every $t \in [t_{\=},T)$, $Cl(M(t) \wo M_{\=}(t))$ is two (path-) connected components. 
		The same can be said for $\W(t) \wo Cl(\W_{\=}(t))$. 
\end{lem}

(See again Figure~\ref{fig:neck})

\begin{proof}
	By definition of $t_{\=}$, $M(t)$ has a neck for $t \in [t_{\=},T)$. 
	Note specifically that $N=2$.
	Since the flow preserves embedding, $M(t)$ is simply connected for $t \in [0,T)$. 
	Since $M(t) \cap K(t)$ is homotopic to a circle in $M(t)$, the result follows from the Jordan Curve Theorem for spheres.
	\\

	Now for $\W(t) \wo \W_{\=}(t)$. 
	We know $\W(t)$ is simply connected. 
	Since $\W(t) \wo Cl(\W_{\=}(t))$ is open, it is locally path-connected. 
	Were $\W(t) \wo Cl(\W_{\=}(t))$ one component, then $\W(t)$ would not be simply connected, since we could construct a nontrivial path in $\W(t)$ that passes through $\W_{\=}(t)$. 
	That is a contradiction. 
	Considering \thref{rk:lid}, $\bd \W_{\=}(t) \cap \bd \left( \W(t) \wo \W_{\=}(t) \right)$ is two components, the ``lids'', because of the way $K(t)$ is constructed. 
	Then any component of $\W(t) \wo Cl(\W_{\=}(t))$ is connected to one of the lids. 
	Thus, there cannot be any more than two components to $\W(t) \wo Cl\left( \W_{\=}(t) \right)$.

\end{proof}

\begin{defn}
	\thlabel{defn:bulb}
	For $t \in [t_{\=},T)$, consider the two components of $M(t) \wo M_{\=}(t)$.
		Call them $M_{\rb}(t)$ and $M_{\lb}(t)$, chosen so that $x_2(Cl(M_{\lb}(t)) \cap K(t))<0$ and $x_2(Cl(M_{\rb}(t)) \cap K(t))>0$.
		Similarly define $\W_{\lb}(t)$ and $\W_{\rb}(t)$ to be the two components of $\W(t) \wo Cl(\W_{\=}(t))$ with $x_2(\W_{\lb}(t) \cap K(t)) < 0$ and $x_2(\W_{\rb}(t) \cap K(t)) > 0$.
		We may use ``bulbs'' to refer to $M_{\lb}(t)$ and $M_{\rb}(t)$ \emph{or} 
		$\W_{\lb}(t)$ and $\W_{\rb}(t)$. 
	It should be clear from context whether we mean the surface or its interior region. 
	\\
\end{defn}

\begin{figure}[h]
	\centering
	\includegraphics[scale=.3,trim=0 3cm 0 3cm,clip]{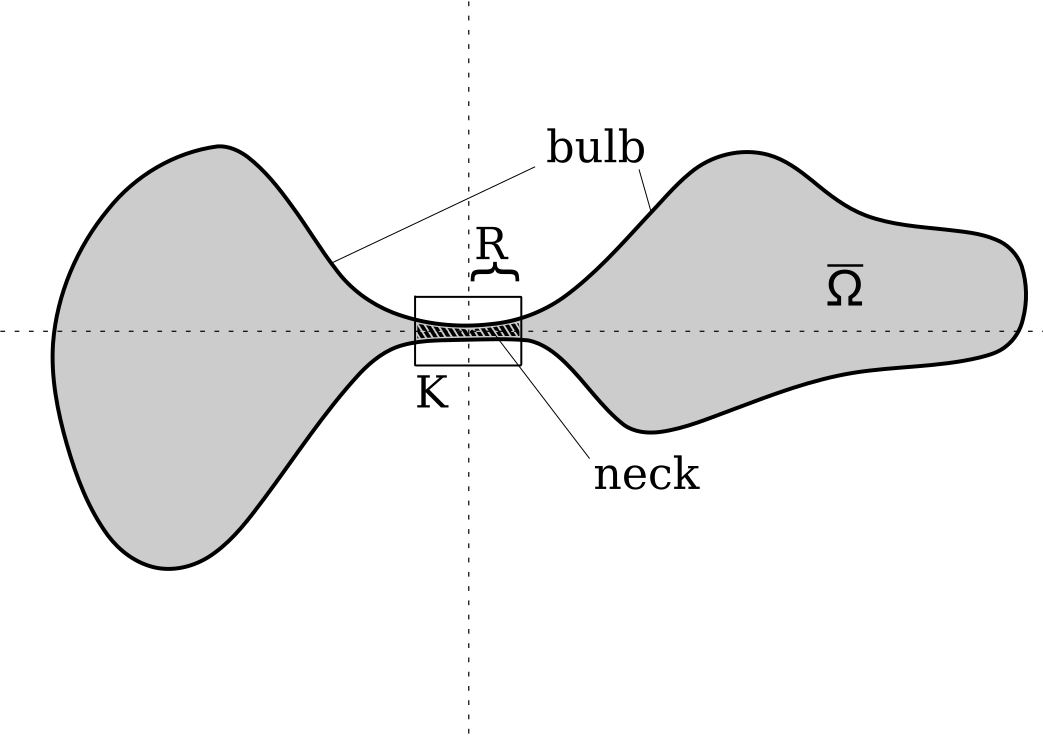}
	\caption{Neck in $K(t)$.}
	\label{fig:neck}
\end{figure}

%

Recall the goal is to find, for each bulb, a $p \in \M$ for which $H\left(\bF(p,t)\right)$ stays bounded. 
However, in order to do that, we need to know that the bulbs do not collapse to the origin. 
Even though each bulb has points outside the neck throughout the flow, either bulb could shrink into the origin at a rate slower than the neck. \\

Once we know the existence of limit bulbs outside the origin, then we can address their regularity. 
\\

\begin{defn} \thlabel{defn:limbulb}
	Define $M_{\rb}^*$ to be the set of points $p^*$ for $p \in M_{\rb}(t_{\=})$. 

	Define $\W^* := \bigcap_{t \in [0,T)} \W(t)$ and $\W_{\rb}^* := \bigcap_{t \in [t_{\=},T)} \W_{\rb}(t)$. 

		Define $M_{\lb}^*$ and $\W_{\lb}^*$ similarly.

\end{defn}

\begin{rk}
	\thlabel{rk:origin}
	Note $M_{\rb}^*$ will always include the origin. 
		Thus, it is no problem if $\bF(p,t)$ becomes trapped in the neck, since this means $p^*=0 \in M_{\rb}^*$. 

\end{rk}

\begin{lem} \thlabel{lem:bulb}
	Neither bulb collapses to the origin. 
	That is, there is $p \in \M$ so $p^* \in M_{\rb}^*$ and $p^* \neq 0$. 
	The same is true for $M_{\lb}^*$
\end{lem}

\begin{proof}
	Let $p \in \M$ so that 
	\[
		x := \bF(p,t_{\=}) \in M_{\rb}(t_{\=}) .
	\]
	By definition of $M_{\rb}(t_{\=})$, $|x| \ge x_2(x)>4\l^{-1}(t_{\=})$. 
	By \thref{lem:TI}, $|p^*-x| \le \l^{-1}(t_{\=})$. 
	\\

	Therefore $p^* \ge 3 \l^{-1}(t_{\=}) > 0$. 
\end{proof}

	Now that we've established that $M_{\rb}^*$ and $M_{\lb}^*$ have some points left to  work with, we set out to make sure each limit bulb has at least one regular point. 
	We do so in the next subsection. 
	\\
%
%
%
%
%



\subsection{Bulbs Do Not Collapse} \label{sec:bulb}
In this subsection, we show that neither bulb can have an entirely singular limit set. 
Thence we conclude the preimage of each bulb has a point $p$ for which $H\left( \bF(p,t) \right)$ stays bounded. 

	\begin{subl} \thlabel{subl:noloop}
		The limit set $M^*$ is simply connected.
	\end{subl}

	\begin{proof}
		First, since $M(t)$ is an embedding of $\M$ for each $t \in [0,T)$, $M(t)$ remains simply connected. \\

		The set $\W_0 \wo \W^*$ is foliated by $M$ and each $M(t)$ is embedded.
			Therefore, for each $x \in \W_0 \wo \W^*$, we can find $(p,t) \in \M \times [0,T)$ so that $\bF(p,t) = x.$ Then define $x^* := p^*$. \\

				Since each $M(t)$ is embedded and the map $p \mapsto p^*$ is continuous (Lemma 2.6 of~\cite{stn}), the map $x \mapsto x^*$ defines a (continuous) retraction from $\W_0 \wo \W^*$ to $M^* \sub \W_0 \wo \W^*$.
		Therefore $M^*$ is simply connected. 
	\end{proof}

	Next we show that each bulb has a regular point. 
	Intuitively, if the whole bulb becomes singular, then its limit set is a curve. 
	Thus, we can choose a ``farthest'' point of $M_{\rb}^*$ (or $M_{\lb}^*$), which must be cylindrical, but that would be strange, since it would suggest there were even farther points on the other end of the neck. 
	That will lead to a contradiction. 

	\begin{lem} \thlabel{lem:coll}
		Neither $M_{\rb}^*$ nor $M_{\lb}^*$ is entirely singular. 
	\end{lem}

	\begin{proof}
		Suppose that $M_{\rb}^*$ is entirely singular. 
		From~\cite{cms}, we know that the set of singular points of the flow at time $T$ is locally representable by a Lipschitz map over a line. 
		We take from this that, since $M_{\rb}^*$ is entirely singular, $M_{\rb}^*$ is a simple curve. 
		By \thref{lem:bulb}, that curve is not a singleton. 
		By \thref{subl:noloop}, that curve is not closed. 
		\\

		Recall that by \thref{rk:origin}, $M_{\rb}^*$ must also contain the origin. 
		Thus, by \thref{subl:noloop}, and since $M_{\rb}^*$ is compact, the origin is one endpoint of the curve, and it has another endpoint that is not the origin or in $M_{\lb}^*$.
		That other endpoint, call it $x$, attains the maximum intrinsic distance from 0 in $M_{\rb}^*$. 
		\\

		By the opening supposition, $x$ is a singular point, which we have also assumed to be cylindrical. 
		This leads us to a contradiction, since by \thref{lem:bulb}, $x$ has its own left and right bulbs, so there are farther points than $x$ from 0 along the curve $M_{\rb}^*$. 
		Thus, $M_{\rb}^*$ is not entirely singular. 
		The same argument applies to $M_{\lb}^*$. 
	\end{proof}

	\begin{cor}
		\thlabel{cor:reg}
		There is at least one point $x \in M_{\rb}^*$ with a point $p \in \M$ such that $p^*=x$ and $H(\bF(p,t))$ stays bounded as $t \nearrow T$. 
		The same is true of $M^*_{\lb}$. 
	\end{cor}

	\begin{proof}
		By \thref{lem:TI}, for every $p \in \M$, $p^* \in M^*$ exists, and $M^*$ is composed entirely of such points $p^*$. 
		By \thref{lem:coll}, there is a $p^*$ (with corresponding $p \in \M$) for each bulb of $M^*$ at which $H(p^*)$ is defined and finite. 
		By the local continuity of the flow, $H\left( F(p,t) \right)$ must stay bounded.

	\end{proof}

	\section{Continuity of Singular Time} \label{sec:mc}

Here we set out to prove our main result. \\

The case where $M_{n0} \sub \BW_0$ offers much more control over the behavior of $M_n$ in terms of the behavior of $\BM$.
Recall in this case, because of well-posedness, we only need to show that $T_n \le \BT$ for large $n$.
Once we ensure the conclusion of \thref{thm:mc} holds in that case, we use an argument like that of \thref{thm:T} to finish the proof of \thref{thm:mc}, addressing the limit of $T_n$ more generally. 
\\

We break the work into two propositions, corresponding to the nonsimply and simply connected cases. 
We dispense with the nonsimply connected case first, since it is much simpler. 

\begin{prop} \thlabel{prop:nonsimply}
	Let $\BM_0$ be a smoothly embedded, closed surface. 
	Let $M_{n0}$ be a sequence of smoothly embedded, closed surfaces such that $M_{n0} \to \BM_0$ and $M_{n0} \sub \BW_0$.
	Assume $\BM_0$ is not simply connected, $\BM$ is type-I, and that $ \BM$ has a cylindrical singularity at time $\BT$.

	Then there is an $n_0>0$ so that $T_n \le \BT$ whenever $n>n_0$.
\end{prop}

(For related illustrations, see Figures~\ref{fig:K} and~\ref{fig:nonsimply}.)

\begin{proof} 
	Assume, without loss of generality, that the cylindrical singularity in question is at the origin. 
	Assume the axis of the cylinder is the $x_2$-axis, and $t \in [t_{\=},\BT)$, so we can make use of $\DD(t)$, which we can do by \thref{lem:neck}. \\

	Recall that mean curvature flow is well-posed and $\BM(t)$ is compact for $t \in [0,\BT)$. 
		Then there is $n_0$ after which $M_n(t_{\=})$ is a graph over $\BM(t_{\=})$, so $M_n(t_{\=}) \cong \BM(t_{\=})$. 
	So assume $n>n_0$.
			\\

	Since $\BM(t)$ is embedded for each $t \in [0,T)$, $\BM(t)$ remains nonsimply connected. 
	Since $\BM(t_{\=})$ is not simply connected, neither is $M_n(t_{\=})$. 
	Choose a curve $\g_{n,t_{\=}} \sub M_n(t_{\=})$ that is not contractible to a point within $M_n(t_{\=})$ and passes through $\DD$ exactly once. 
	Thus $\g_{n,t_{\=}}$ forms a Hopf link with $\bd\DD(t_{\=})$. \\

	We would like to subject $\g_n(t) \sub M_n(t)$ to the flow of $M_n$, with initial condition $\g_{n,t_{\=}}$ at time $t_{\=}$. 
	So for $t \in (t_{\=},\BT)$, define 
\[
	\g_n(t) := \bF_n(\bF_n^{-1}(\g_n,t_{\=}),t) .
\]

		By \thref{lem:neck}, $\BM(t) \cap \DD(t)$ remains a closed curve.
		Furthermore, since the flow preserves the condition $M_n(t) \sub \BW(t)$, we have that $M_n(t) \cap \DD(t)$, is bounded away from $\bd \DD(t)$ for $t \in [t_{\=},\BT)$, so $\g_n(t) \cap \DD(t)$ is also bounded away from $\bd \DD(t)$ (not uniformly in $t$).
		Since the curve $\bd \DD(t)$ shrinks homothetically, and mean curvature flow is continuous, both $\bd \DD(t)$ and $\g_n(t)$ undergo homotopy. 
		Therefore, the Hopf link formed by $\g_n(t_{\=})$ and $\bd \DD(t_{\=})$ is preserved until time $\BT$.
		\\

	In the same vein, since $M_n(t) \sub \BW(t)$, each component of $M_n(t) \cap \DD(t)$ is a closed curve. 
	Also, because of the Hopf link, for every $t \in [t_{\=},\BT)$ there is at least one point in $\DD(t)$ at which $\g_n(t)$ intersects $\DD(t)$ transversely. 
		Thus, for each $t \in [t_{\=},\BT)$, at least one of those curves is not a singleton. 
			Since $\DD(t)$ is a disk with radius $\l^{-1} \xrightarrow[t \to \BT]{} 0$, the maximum curvature on $M_n(t) \cap \DD(t)$ blows up no later than time $\BT$. 
		\begin{figure}[h]
	\centering
	\includegraphics[scale=.2]{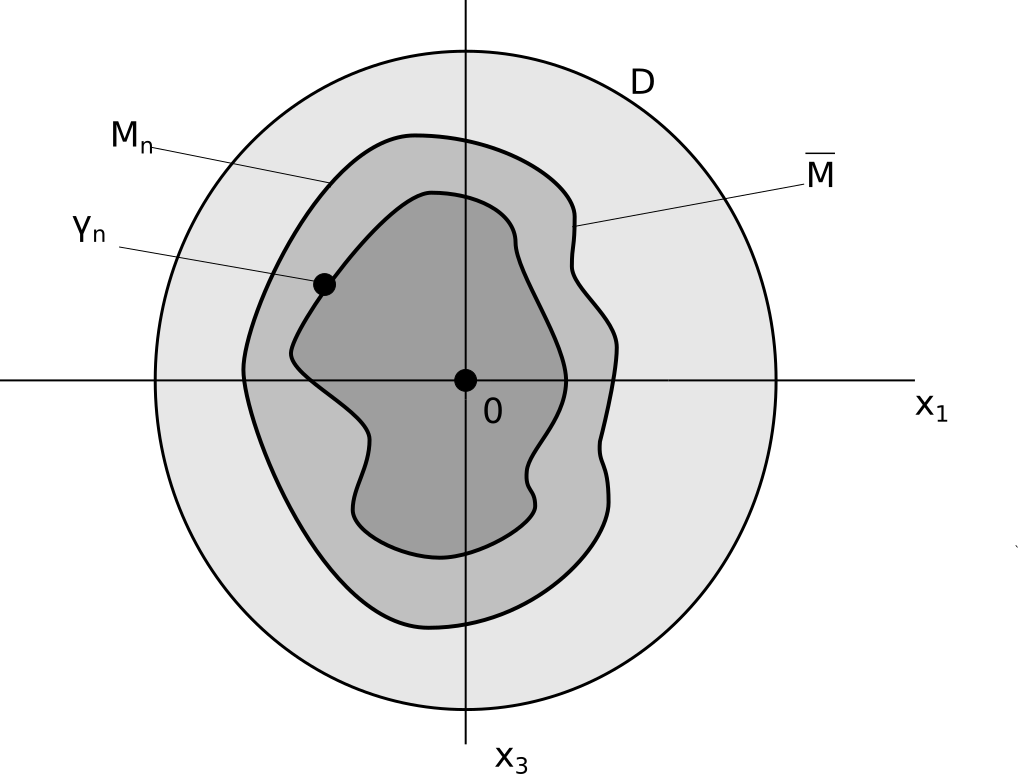}
		\caption{}
		\label{fig:nonsimply}
	\end{figure}
\end{proof}
 
Now for the case where $\BM$ is simply connected. 

\begin{prop} \thlabel{prop:contain}
	Let $\BM_0$ be a smoothly embedded, closed, mean-convex surface. 
	Let $M_{n0}$ be a sequence of smoothly embedded, closed surfaces such that $M_{n0} \to \BM_0$ and $M_{n0} \sub \BW_0$.
	Assume $\BM_0$ is simply connected, $\BM$ is type-I, and that $ \BM$ has a cylindrical singularity at time $\BT$.

	Then there is an $n_0>0$ so that $T_n \le \BT$ whenever $n>n_0$.
\end{prop}

The proof is a rather technical procedure, so we begin with some motivation. 
Recall the hope is to construct something like that in Figure~\ref{fig:ball}. 
For each sphere, we need to choose the radius $r$, the time $t_0$ at which to place the sphere, $n$ so that $M_n(t_0)$ is close to $\BM(t_0)$, and the point $y \in M_n(t_0)$ at which we place the sphere (see Figure \ref{fig:distance}) .
When placing the sphere in $M_n(t_0)$, we have three concerns. 
\begin{enumerate}[(i)]
	\item \label{it:fit}	\emph{The sphere must fit in $M_n(t_0)$.}
		Here, we turn to the regularity results of \S \ref{sec:bulb} to apply the Andrews condition.
		In addressing this, we prescribe a maximum radius for the sphere, thus fixing its lifespan. 
	\item \label{it:neck} \emph{The sphere should not intersect the neck.} We want that the sphere stays out of $K$, keeping some points of $M_n(t_0)$ away from the neck. 
		(The neck and sphere will contract away from each other, so this condition is preserved.)
	\item \label{it:live} \emph{The sphere must outlive the neck.} Given the fixed lifespans of $\BM$ and the sphere (once $r$ is chosen), we need only wait to place the sphere until it will live past $\BT$.
\end{enumerate}

The preceding conditions are mostly about $M_n(t_0)$, but we only have control over $M_n(t_0)$ via $\BM(t_0)$ by well-posedness. 
This makes dependencies more delicate. 
Therefore the proof is broken into three parts: choosing $r$, then $t_0$, then $n$ and $y$. 
\\

In the first part, we choose the radius $r$ small enough to facilitate (\ref{it:neck}) and (\ref{it:fit}). 
In the second part, we choose $t_0$ close enough to $\BT$ that (\ref{it:live}) is satisfied and $R(t_0)$ satisfies (\ref{it:neck}) (intuitively, we are waiting for $\BM(t)$ to develop a neck very small compared to the bulbs).
In the third part, we choose $n$ large enough that $M_n(t_0)$ approximates both bulbs, so a sphere of radius $r$ can be placed in each of $\W_{n\rb}(t_0)$ and $\W_{n\lb}(t_0)$.
\\

\begin{proof}[Proof of \thref{prop:contain}]

	Assume, without loss of generality, that the cylindrical singularity in question is at the origin. 
	Assume the axis of the cylinder is the $x_2$-axis, so we can make easy use of $K(t)$. \\

	Assume $t \ge t_{\=}$, so $\BM$ has a neck.
For simplicity, we'll do the proof just in terms of the right bulb. 
\\

\begin{sloppypar}
	By \thref{cor:reg}, there are $x \in M_{\rb}^*$, $p \in \M$, and $C>0$ so that ${\BF(p,t) \xrightarrow[t \to \BT]{} x \neq 0}$ and $\BH(\BF(p,t))<C$ for $t \in [t_{\=},\BT)$. 
		Several choices of constants and objects in the proof rely on careful spacing with respect to $x$ and the origin. 
		To that end, the quantity $\d=\frac{|x|}{8}$ is convenient. 
		(See Figure \ref{fig:outof} for a preview)
	\end{sloppypar}

	\paragraph{Part I: Choosing $r$} 
	Since $\BM_0(t)$ becomes strictly mean convex immediately, there is some $\a>0$ so that $\BM(t_{\=})$ is $2\a$-non-collapsed (so $M_n(t_0)$ will be $\alpha$-non-collapsed when we choose it). 
	Take
		\[
			r_1 = \frac{\a}{2C}
		\]
	as an upper bound for $r$.
	Now choose $r = \min\{r_1,\d\}$. 
	\\

	Later, $r \le \d$ will help us with (\ref{it:neck}) (see again Figure \ref{fig:outof}), 
	and $r \le r_1$ will allow us to use well-posedness to help with (\ref{it:fit}) (we cannot choose \emph{where} to place the sphere until after we have chosen $t_0$, $n$, and $y$). 
		\\

	\paragraph{Part II: Choosing $t_0$}
	Since $|\BF(p,t)| \xrightarrow[t \to \infty]{} x$, there must be $t_1 \in [t_{\=},\BT)$ after which $|\BF(p,t)|\ge \frac{|x|}{2}$, by continuity of the flow. 
		Let $R(t)=4 \l^{-1}(t)$, the radius of $K(t)$.
		Then $R(t)$ shrinks to 0 by time $\BT$. 
		Thus there is a time $t_2 \in [t_1,\BT)$ after which $R(t) \le \frac{|x|}{16}$ (See Figure \ref{fig:outof}). 
			These two conditions will help with (\ref{it:neck}) in part III. 
			\\

	\begin{figure}[h]
		\centering
		\includegraphics[scale=.4]{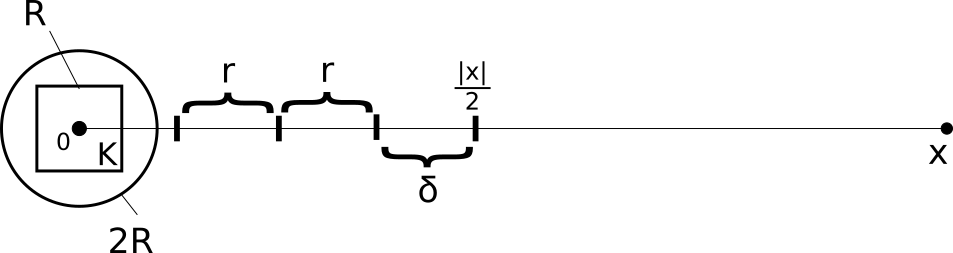}
		\caption{``Worst case scenario'', with distances aligned on the $x_2$-axis}
		\label{fig:outof}
\end{figure}

	Since the radius $r$ of the sphere is already fixed, we know its lifespan. 
	Call it $\t$. 
	We can find the time $t_3 \in [t_2,\BT)$ at which $\t=2(\BT-t_3)$. 
	This means if the sphere begins its flow at any time in $[t_3,\BT)$, the sphere will survive past time $\BT$.
	\\

	Choose $t_0=t_3$, so $t_0$ has the properties of $t_1,t_2,t_3$.
	Then, because of the choice of $t_3$, we have already satisfied (\ref{it:live}).

	\paragraph{Part III: Choosing $n$ and $y$}

	Let $x_0=\BF(p,t_0)$, so $|x_0| \ge \frac{|x|}{2}$.
	\\

By well-posedness, there is $n_1$ such that if $n\ge n_1$, $M_n(t_0)$ is $\a$-non-collapsed. 
	Given $\d>0$ above, and recalling $\BH(x_0)=\BH(\BF(p,t_0))<C$, there exists $n_2\ge n_1$ so large that, if $n \ge n_2$, then there is a point $y \in M_n(t_0)$, within $\d$ of $x_0$, so that $H_n(y) \le 2C$. 
	Let $n_0=n_2$. 
	Now assume $n \ge n_0$ so that $M_n(t_0)$ is $\a$-non-collapsed and that such a $y$ exists. 
	Choose that $y$. 
	\\

	We now have the following:
		\[
			\begin{array}{c}
				|y| \ge |x_0|-\d \ge \frac{|x|}{2} - \d = 3\d \\
				R(t_0) \le \frac{\d}{2} \\
				\mbox{sphere diameter} = 2r \le 2 \d .
			\end{array}
		\]

		Those together imply that, were a sphere of radius $r$ placed at time $t_0$ touching $y$, the distance between sphere and $K(t_0)$ is at least $\frac{\d}{2}$.
		Since the sphere would contract under mean curvature flow, and $K(t)$ contracts by definition, they would stay disjoint. 
		Thus if we flow the sphere by mean curvature flow, as we do with $M_n(t)$, we are done with (\ref{it:neck}).
		(See Figures~\ref{fig:distance}~and~\ref{fig:outof})
		\\
		
		\begin{figure}[h]
	\centering
	\includegraphics[scale=.4]{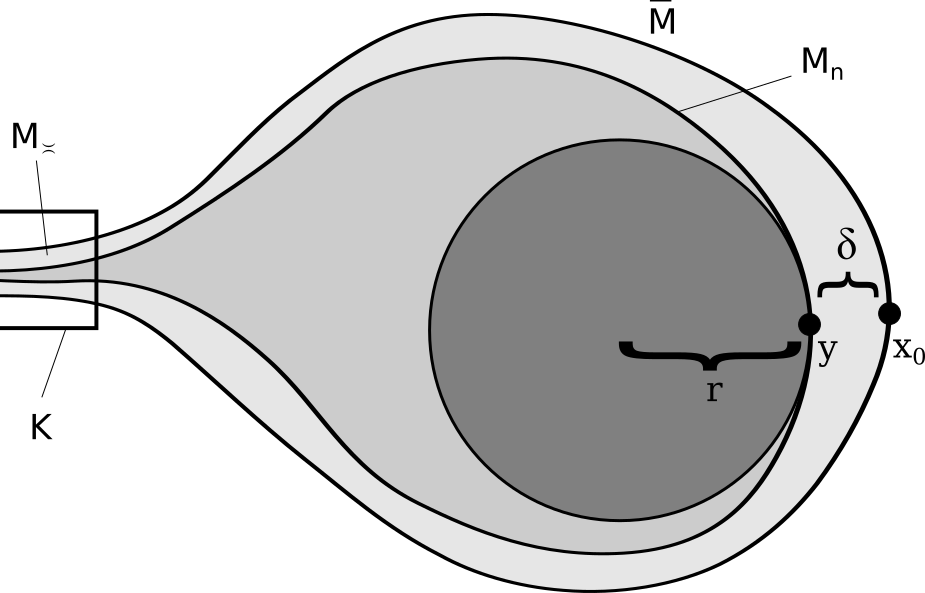}
	\caption{Sphere fits in bulb far from neck}
		\label{fig:distance}
	\end{figure}

Recall $M_n(t)$ is $\a$-non-collapsed. 
Then, since $H_n(y) \le 2C < \infty$, and $r \le r_1 = \frac{\a}{2C}$, there is room to place a sphere of radius $r$ inside $Cl(\W_n(t_0))$, tangent at $y$. 
Since the sphere is disjoint from $K(t_0)$, it is contained in $Cl(\W_{n\rb}(t_0))$. 
That takes care of (\ref{it:fit}). \\

Thus, for our choice of $r$, $t_0$, $n$, and $y$, (i)-(iii) are all satisfied, with regards to the right bulb.

\paragraph{Finishing the proof}
		Repeat the above argument for $\BM_{\lb}(t)$. 
		\\

		Since there is a sphere in each bulb of $\W_n(t)$, up to time $\BT$, there are points of $M_n(t)$ in each bulb of $\BW_n(t)$ up to time $\BT$. 
		Therefore, if $T_n > \BT$, then $M_n(t)$ has a nontangential intersection with $\DD(t)$ for $t \in [t_0,\BT)$. 
			As in the proof of \thref{prop:nonsimply}, since $\BM(t) \cap \DD(t)$ is a closed curve, and $M_n(t) \sub \BW(t)$, at least one component of $M_n(t) \cap \DD(t)$ is a nontrivial closed curve for all $t \in [t_0,\BT)$. 
			Since $\DD(t)$ collapses to a point at time $\BT$, the maximum curvature on $M_n(t) \cap \DD(t)$ must blow up. 

\end{proof}

Propositions \ref{prop:nonsimply} and \ref{prop:contain} completely cover the case where $M_{n0} \sub \BW_0$. 
So we mimic the proof of \thref{thm:T} to reduce the proof of \thref{thm:mc} to that case. 

\begin{proof}[Proof of \thref{thm:mc}]
	Let $\BM$ and $M_n$ be as in \thref{thm:mc}.
	By well-posedness, we already have that $\dsp \liminf_{n \to \infty} T_n \ge \BT$. 
	So we need only show that $\dsp \limsup_{n \to \infty} T_n \le \BT$. 
	\\

	Let $0<\varepsilon$. 
	Define $\hM_n(t)=M_n(t+\e)$.
	We need $\hT_n > \e$ for the following proof to make sense (since the intervals of existence times for $\BM$ and $\hM_n$ need to overlap).
	However, the goal at the end of the proof is to show that, for large $n$, $\hT_n < \BT+\e$, so if $\hT_n \not>\e$, we are done. 
	So we can just assume $\hT_n > \e$. 
	\\

	Now $\hM_{n0} = M_n(\varepsilon)$. 
	Since $\BM(t)$ strictly is mean-convex for $t=[\frac{\varepsilon}{2},\BT)$, its velocity at every point is inward with positive speed. 
Thus $\hM_0(\e) \sub \BW_0$, and we have the Hausdorff distance $d=d_H \lp \BM(\e),\BM_0 \rp > 0$. 
By well-posedness, there is an $n_0>0$ so that if $n \ge n_0$, then $d_H(\BM(\e),\hM_{n0})<\frac{d}{2}$ (see \thref{lem:fth}).
So assume $n \ge n_0$.
Rearranging
	\[ 
		d_H(\BM(\e),\BM_0) \le  d_H(\BM(\e),\hM_{n0}) + d_H(\hM_{n0},\BM_0) 
	\]
	gets us
	\[
		d_H(\hM_{n0},\BM_0) 
		\ge d_H(\BM(\e),\BM_0) - d_H(\BM(\e),\hM_{n0}) 
		> d-\frac{d}{2} 
		=\frac{d}{2} 
		> 0 .
	\]
Thus $\hM_{n0} \sub \BW_0$. \\

We turn to \thref{cor:sing} to see that $\BM$ must shrink to a point at time $\BT$ or have a cylindrical point at time $\BT$. 
In the former case, apply \thref{thm:T}.
In the latter case, we need to apply \thref{prop:nonsimply} or \thref{prop:contain} accordingly.
However, \thref{prop:contain} requires connectedness. 
Consider the component of $\BW_0$ that contains the singular point, then restrict all attention to its boundary. 
It is sufficient to apply the two propositions to that component (and the corresponding component of $M_{0n}$). \\

Then we see that $T_n = \hT_n + \e  \le \BT + \e$.
Since that is true for any $n \ge n_0$, we have $\dsp \limsup_{n \to \infty} T_n \le \BT + \e$.
Since $\e$ was arbitrary, we are done. 
Then either \thref{prop:nonsimply} or \thref{prop:contain} applies. 

\end{proof}

\section{Continuity of the Limit Set} \label{sec:profile}

Recall the Hausdorff distance 
\[
		d_H(X,Y) = \max \left\{ \sup_{x \in X} \inf_{y \in Y} |x-y|, \sup_{y \in Y} \inf_{x \in X} |y-x| \right\} ,
\]
and that if $\hS$ is a graph of $f$ over $\Sigma$, then by \thref{lem:fth}
\[
	d_H(\Sigma,\hS) \le \|f\|_{C^)} .
\]

\begin{rk}
	Although the case where $\BM$ contracts to a point can be made very simple with an argument similar to that in the proof of \thref{thm:T}, the proof below suffices for both spherical and cylindrical cases. 
\end{rk}

\paragraph{Proof of \thref{cor:profile}}

For intuition, note from parabolic regularity (under the assumption that $T_n > \BT$), it makes sense that 
	\[
		M_n(\BT) \sim M_n(t) \sim \BM(t) \sim \BM(\BT) 
	\]
for a fixed $t$ close to, but less than, $\BT$.
	We need \thref{thm:T} to assert that $M_n^* = M_n(T_n)$ is anything like $M_n(\BT)$. 
\begin{proof}
	Let $\e>0$, and set $C =\sqrt{2N}$. 
	Let $d_H$ denote Hausdorff distance. 
	\thref{lem:fth} says we can make $d_H\left( M_n(t_0),\BM(t_0) \right)$ small by making $n$ large. \\
	
	\noindent Choose: 
	\begin{itemize}
		\item $t_0 \in [0,\BT)$ so that $|\BT - t_0| < \e$. \\
			\item 		$n_1$ so $n \ge n_1$ implies $|\BT-T_n|<\frac{\e}{2}$ (which exists by \thref{thm:T} or \ref{thm:mc}). \\
				\indent (Keeps $T_n$ close to $\BT$, which also means $T_n>t_0$.) \\
			\item $n_2$ so $n \ge n_2$ implies $d_H(M_n(t_0),\BM(t_0)) < \e$ (which exists by well-posedness). \\
		\end{itemize}

	\noindent
	By \thref{lem:S}, every point of $\BM(t_0)$ is within $C\sqrt{\BT-t_0}$ of $\BM^*$. 
	Therefore, we have
	\[
		d_H \lp \BM(t_0), \BMs \rp \le C \sqrt{\BT-t_0} \le C \sqrt{\e}
	\]
	and
	\begin{align*}
		d_H(M^*_n,M_n(t_0)) 
		& \le C \sqrt{T_n-t_0} 
		= C \sqrt{(T_n-\BT) +(\BT-t_0)} \\
		& \le C \lp \sqrt{|T_n-\BT|} + \sqrt{|\BT-t_0|} \rp 
		< 2 C \sqrt{\e} .
	\end{align*}

	Now assume $n \ge \max \{n_1,n_2\}$, and apply 
		\begin{align*} &d_H(M^*_n, \BMs)  \\
			\le & d_H(M^*_n,M_n(t_0))   +   d_H(M_n(t_0),\BM(t_0))   +   d_H(\BM(t_0),\BMs) \\
			\le & 2 C \sqrt{\e} + \e + C \sqrt{\e} . 
		\end{align*}
Since $\e$ was arbitrary, we are done. 
\end{proof}

	\bibliographystyle{plain}
\bibliography{bib}

\end{document}